\documentclass[11pt]{article}

\usepackage[T1]{fontenc}
\usepackage[utf8]{inputenc}
\usepackage{lmodern}
\usepackage{lineno}
\usepackage{abstract}
\usepackage{geometry}
\usepackage{mathrsfs}
\usepackage{amsmath,amssymb,amsthm,mathtools}
\usepackage{bm}
\usepackage{float}      
\usepackage{placeins}   

\usepackage{xcolor}
\definecolor{linkblue}{RGB}{0,60,180}

\usepackage{hyperref}
\hypersetup{
  colorlinks=true,
  linkcolor=linkblue,
  citecolor=linkblue,
  urlcolor=linkblue,
  pdftitle={Extensions of Real-Weighted Fractional Arboricity},
  pdfauthor={Rowan Moxley},
  pdfsubject={Graph Theory},
  pdfkeywords={arboricity, weighted graphs, effective resistance, Foster's theorem, Rayleigh monotonicity}
}

\usepackage[nameinlink,noabbrev]{cleveref}

\usepackage{enumitem}
\setlist{nosep}

\usepackage{tikz}
\usetikzlibrary{positioning,calc,arrows.meta,fit,backgrounds}
\usepackage{caption}
\usepackage{subcaption}

\numberwithin{equation}{section}

\theoremstyle{plain}
\newtheorem{theorem}{Theorem}[section]
\newtheorem{lemma}[theorem]{Lemma}
\newtheorem{proposition}[theorem]{Proposition}
\newtheorem{corollary}[theorem]{Corollary}
\newtheorem{example}[theorem]{Example}
\theoremstyle{definition}
\newtheorem{definition}[theorem]{Definition}

\theoremstyle{remark}
\newtheorem{remark}[theorem]{Remark}

\newcommand{\Acal}{\mathcal{A}}

\newcommand{\ReffG}{R_{\mathrm{eff}}^{(G)}}
\newcommand{\ReffX}{R_{\mathrm{eff}}^{(X)}}
\newcommand{\ReffQd}{R_{\mathrm{eff}}^{(Q_d)}}
\newcommand{\iext}{i_{\mathrm{ext}}}

\Crefname{theorem}{Theorem}{Theorems}
\Crefname{lemma}{Lemma}{Lemmas}
\Crefname{proposition}{Proposition}{Propositions}
\Crefname{corollary}{Corollary}{Corollaries}
\Crefname{definition}{Definition}{Definitions}
\Crefname{remark}{Remark}{Remarks}
\Crefname{equation}{Equation}{Equations}
\Crefname{section}{Section}{Sections}

\title{Extensions of Real-Weighted Fractional Arboricity}

\author{Rowan Moxley\\[0.5em]
  Department of Mathematics and Statistics\\
  Eastern Michigan University\\
  \texttt{rmoxley@emich.edu}}

\date{December 9, 2025}

\begin{document}
\maketitle
\begin{abstract}
We study a conductance-weighted arboricity $\Acal_c(G)$ for a finite simple undirected graph
$G=(V,E,c)$ with a conductance assignment $c:E\to(0,\infty)$. This functional reduces to the fractional arboricity $\gamma(G)$ when $c\equiv 1$, is isomorphism invariant, monotone under taking subgraphs and adding edges, positively homogeneous, and convex. We also prove sharp global bounds with attainment at a connected subgraph. On the analytic side, we introduce a local variant and derive conductance--resistance inequalities using effective resistances in the ambient network,
which in turn yields an upper bound and hence an explicit effective resistance-based upper bound on $\Acal_c(G)$. On the structural side, we describe the algebraic behavior of $\Acal_c(G)$. We show that under
edge-disjoint unions of graphs, $\Acal_c(G)$ behaves as a max invariant: for a finite disjoint union of weighted
graphs one has $\Acal_c(G)= \max_i \Acal_{c_i}(G_i)$. In particular, disjoint union induces a commutative
idempotent monoid structure at the level of isomorphism classes, with $\Acal_c(G)$ idempotent with
respect to this operation. We also provide a computational exhibit on the hypercube family $Q_d$, including random conductance sampling, illustrating numerical evaluation of the resulting resistance-based bound.

\end{abstract}


\section{Introduction}
\label{sec:intro}
The arboricity of a graph $G$, defined as the minimum number of forests whose union covers
every edge of $G$, admits a classic optimization characterization due to Nash--Williams
\cite{nashwilliams1964}:
the arboricity is the ceiling of the maximum, over all connected subgraphs $H$ with $|V(H)|\ge 2$,
of $|E(H)|/(|V(H)|-1)$. The \textit{fractional arboricity}, denoted $\gamma(G)$, is the real-valued analogue resulting from the simple omission of the ceiling function from the classical case; this interpretation has proven to be useful in
structural and algorithmic graph theory
  \cite{gabow1992,picard1982,guruswami2003}.

Weighted analogues of arboricity have been studied in several directions. When edges carry
positive integer weights, Gabow and Westermann \cite{gabow1992} gave min-max theorems for the minimum number of forests needed to
cover each edge at least its weight. Real-weighted versions, especially with weights
interpreted as capacities or costs, appear in multicommodity flow and congestion minimization
\cite{chekuri2013}. However, the natural \emph{conductance-weighted}
analogue, where the density numerator is the sum of conductances (reciprocal resistances)
rather than cardinality or integer weights, has received surprisingly little direct attention. Many intended applications of this paper benefit from real-valued parameters (such as covariates in regressions), so we intentionally drop the ceiling in Nash–Williams’ definition and move from integer- to real-valued weights, using conductances due to their many useful qualities; these qualities are well-understood in the language of electrical networks, which has become a standard tool for studying
graph densities. Foster's theorem \cite{foster1949} and Rayleigh's monotonicity principle \cite{doylesnell1984}
imply that the effective resistance $R^{(G)}_{\text{ eff }(e)}$ between any two vertices is at most the
reciprocal of the number of edge-disjoint paths, and this perspective has been used to bound
spanning-tree densities and Laplacian matrix eigenvalues
  \cite{bollobas1998}. 

Despite the close formal analogy between
conductance-weighted spanning-tree quantities and effective resistance, few systematic studies relating arboricity to effective resistance exist within the literature. Most similar to this paper, effective resistances have been used to bound fractional arboricity in the unweighted setting, and in fact a bound which takes similar structure to our weighted arboricity functional is procured in Zhou et al. \cite{zhou2021edge} by using Rayleigh's monoticity principle. Our work diverges from Zhou et al.'s bound on fractional arboricity in that it studies a fundamentally different object, a real-weighted functional with exogenous edge weight assignments, and concerns itself on proving various properties about that functional rather than using electrical network theory 
to bound the unweighted fractional arboricity functional; in particular, we derive a local Foster inequality for arbitrary conductance assignments and a resistance bound, and we complement the theory with computational experiments under heterogeneous conductance sampling.

The purpose of this manuscript is to connect arboricity and electrical network theory due to their surprisingly numerous similarities. For a finite undirected graph $G$
equipped with nonnegative edge conductances $c:E\to(0,\infty)$, we define the
conductance-weighted arboricity
\[
\mathcal A_c(G) := \max\Bigl(
\{D_c(H): H\subseteq G,\ H \text{ connected},\ |V(H)|\ge 2\}
\cup \{0\}
\Bigr)
,
\qquad D_c(H):=\frac{\sum_{e\in E(H)}c(e)}{|V(H)|-1}.
\]
\begin{enumerate}
\item The functional $\Acal_c(G)$ inherits all expected analytic properties from the fractional
case: isomorphism invariance, monotonicity under edge addition and weight increase,
positive homogeneity, convexity, and sharp global bounds
$\max c(e)\le \Acal_c(G)\le\sum c(e)$, with attainment (Section~\ref{sec:basic}).
\item Every connected subgraph $H$ satisfies
$\sum_{e\in E(H)}c(e)R^{(G)}_{\text{ eff }(e)}\le |V(H)|-1$,
yielding an explicit resistance-based upper bound
on $\Acal_c(G)$ (Section~\ref{sec:local}). \medskip

\item Under edge-disjoint union, $\Acal_c(G)$ behaves like a max invariant:
$\Acal_c(G\sqcup H)=\max\{\Acal_c(G),\Acal_c(H)\}$. This induces a commutative idempotent monoid
structure on the set of conductance-weighted isomorphism classes
(Section~\ref{sec:algebra}). \medskip

\end{enumerate}

While the basic analytic properties in Section~\ref{sec:basic} follow the standard approach
for graph density and offer no surprises, the resistance inequality in
Section~\ref{sec:local} and especially the monoid observation in
Section~\ref{sec:algebra} appear to be new. In particular, we are unaware of any prior
work that records the idempotent monoid structure for arboricity (weighted or
unweighted) under disjoint union, despite its similarity to the well-known behavior
of girth, independence number, and odd-girth.

The rest of the paper is organized as follows:  Section~\ref{sec:basic} recalls definitions
and proves the basic properties. Section~\ref{sec:local} introduces a local variant
and derives the conductance--resistance inequalities. Section~\ref{sec:algebra} establishes
the monoid structure and Section~\ref{sec:computation} presents a small computational exhibit on families of hypercubes of the form $Q_d$.

\section{Weighted Arboricity: Definitions and Basic Properties}\label{sec:basic}
\subsection{Preliminaries and Notation}
Throughout, graphs are finite, simple, and undirected.  We write
$G=(V,E)$ with $n:=|V|$ and $m:=|E|$.  For $S\subseteq V$, $G[S]$
denotes the induced subgraph on $S$, and for $H\subseteq G$ we write
$V(H)$ and $E(H)$ for its vertex and edge sets.  A subgraph $H$ is
\emph{connected} if it has a single connected component. For weighted graphs we equip $G$ with an \emph{edge conductance}
function $c:E\to(0,\infty)$ and write $(G,c)$; all weights are assumed
nonnegative and are interpreted as conductances.  Definitions of the
weighted density and the conductance-weighted arboricity functional
$A_c(G)$ (and related notation such as $D_c(H)$) are given in the
following section and will be used consistently thereafter.

\subsection{Setup}

\begin{definition}[Weighted density and weighted arboricity]
For any connected subgraph $H\subseteq G$ with $|V(H)|\ge 2$, define the \emph{weighted density}
\[
D_c(H)\;:=\;\frac{\sum_{e\in E(H)} c(e)}{|V(H)|-1}.
\]
The \emph{weighted arboricity} of $(G,c)$ is
\[
\mathcal A_c(G)
:=
\max\Bigl(
\{D_c(H): H\subseteq G,\ H \text{ connected},\ |V(H)|\ge 2\}
\cup \{0\}
\Bigr).
\]
\end{definition}

\begin{remark}[Reduction from disconnected subgraphs]
If one allows disconnected $H$ and replaces the denominator by $|V(H)|-\omega(H)$, where $\omega(H)$ is the number of connected components of $H$, then the resulting maximization equals the one over connected $H$. Indeed, if $H=\bigsqcup_i H_i$ with each $H_i$ connected and $|V(H_i)|\ge 2$, then
\[
\frac{\sum_i \sum_{e\in E(H_i)} c(e)}{\sum_i \bigl(|V(H_i)|-1\bigr)}
\]
is a weighted average of $\{D_c(H_i)\}$ with weights $|V(H_i)|-1\ge 1$, and hence does not exceed $\max_i D_c(H_i)$. Therefore it suffices to optimize over connected subgraphs.
\end{remark}
\begin{remark}[Exclusion of zero edge weights]
    Readers may notice that $0$ was intentionally excluded from the definition of conductance assignment; this is because for the purposes of this research, such an edge weight will contribute nothing to the sum in the numerator of $D_c(G)$. Indeed, excluding $c=0$ edge weights may be interpreted as calculating maxima over connected components of the induced subgraph $G'=(V, \{e\in E(G): c(e)>0\})$, where $\Acal_c(G)=\Acal_c(G')$, so we are left without a sensible reason to include them in the definition of our conductance assignment.
\end{remark}
\subsection{Basic properties}

\begin{proposition}[Reduction to fractional arboricity]\label{prop:reduction}
If $c\equiv 1$ on $E(G)$, then
\[
\Acal_c(G)\;=\;\max_{H\subseteq G \text{ conn}} \frac{|E(H)|}{|V(H)|-1}=\gamma(G),
\]
the fractional arboricity. In particular, $\lceil \Acal_c(G)\rceil$ equals the classical (unweighted) Nash--Williams arboricity.
\end{proposition}

\begin{proof}
With $c\equiv 1$, we have $\sum_{e\in E(H)} c(e)=|E(H)|$, so $D_c(H)=|E(H)|/(|V(H)|-1)$. Taking maxima yields the claim.
\end{proof}

\begin{proposition}[Isomorphism invariance]\label{prop:iso}
If $\varphi:G\to G'$ is a graph isomorphism and $c'$ is transported by $c'(\varphi(e)):=c(e)$, then $\Acal_c(G)=\Acal_{c'}(G')$.
\end{proposition}

\begin{proof}
The map $\varphi$ bijects connected subgraphs, preserves vertex and edge counts, and thus preserves each $D_c(H)$. Taking maxima over corresponding subgraphs gives the result.
\end{proof}

\begin{proposition}[Monotonicity]\label{prop:mono}
If $G'$ is obtained from $G$ by adding edges (with nonnegative weights), then $\Acal_c(G')\ge \Acal_c(G)$. If $c'\ge c$ pointwise on $E(G)$, then $\Acal_{c'}(G)\ge \Acal_c(G)$.
\end{proposition}

\begin{proof}
Adding edges enlarges the feasible family of $H$; $D_c$ on previously feasible $H$ is unchanged. For weight increases, for each fixed $H$, $\sum_{e\in E(H)} c'(e)\ge \sum_{e\in E(H)} c(e)$, hence $D_{c'}(H)\ge D_c(H)$. Taking maxima yields the claim.
\end{proof}

\begin{proposition}[Subgraph monotonicity]\label{prop:subgraph}
If $F\subseteq G$ is any subgraph (with conductances given by restriction of $c$), then $\Acal_c(F)\le \Acal_c(G)$.
\end{proposition}

\begin{proof}
We have two cases. If $|F|< 2$, then $\Acal_c(F)=0$ by definition, and we are done. Otherwise, every connected $H\subseteq F$ is a connected subgraph of $G$ with the same $D_c(H)$. Hence $D_c(H)\le \Acal_c(G)$ for all such $H$. Taking the maximum over $H\subseteq F$ gives $\Acal_c(F)\le \Acal_c(G)$.
\end{proof}

\begin{proposition}[Positive homogeneity and convexity]\label{prop:convex}
For any $\alpha\ge 0$, $\Acal_{\alpha c}(G)=\alpha\,\Acal_c(G)$. Moreover, for any conductance assignments $c^{(1)},c^{(2)}$ and $\lambda\in[0,1]$,
\[
\Acal_{\lambda c^{(1)}+(1-\lambda)c^{(2)}}(G)\ \le\ \lambda\,\Acal_{c^{(1)}}(G)+(1-\lambda)\,\Acal_{c^{(2)}}(G).
\]
\end{proposition}

\begin{proof}
For fixed $H$, $D_{\alpha c}(H)=\alpha D_c(H)$ and $D_{\lambda c^{(1)}+(1-\lambda)c^{(2)}}(H)=\lambda D_{c^{(1)}}(H)+(1-\lambda)D_{c^{(2)}}(H)$. Taking maxima over $H$ yields the statements.
\end{proof}

\begin{proposition}[Global bounds and attainment]\label{prop:bounds}
Let $c_{\max}:=\max_{e\in E(G)} c(e)$, and $C(G):=\sum_{e\in E(G)} c(e)$.
The following hold:
\begin{enumerate}[label=(\roman*)]
\item $\Acal_c(G)\le C(G)$.
\item $\Acal_c(G)\ge c_{\max}$.
\item The maximum in the definition of $\Acal_c(G)$ is attained by some connected subgraph $H^\star\subseteq G$.
\end{enumerate}
\end{proposition}
\begin{proof}
(i) For any feasible $H$, $|V(H)|-1\ge 1$ and $\sum_{e\in E(H)} c(e)\le C(G)$, so $D_c(H)\le C(G)$.\\
(ii) For any spanning tree $T$ of $G$, if $c_{\max}\in T$, then $\Acal_c(G)\ge \Acal_c(T)= c_{\max}$; if $(u,v)=c_{\max}\notin T$, set $H=(u,v)$, then $\Acal_c(G)\ge c_{\max}$ certainly. In either case, the maximum conductance value cannot exceed that of the total fractional arboricity of the graph.\\
(iii) There are finitely many connected subgraphs of $G$; thus the set of densities $\{D_c(H)\}$ contains a maximum.
\end{proof}
\begin{corollary}[Weighted arboricity of trees]\label{cor:trees}
    Let T be a tree with a conductance assignment. Then $\Acal_c(G)=\max_{e\in E(T)}c(e).$
\end{corollary}
\begin{proof}
    Since every connected subgraph of a tree is a subtree, the density \(D_c(H)\) is an average of edge conductances over \(H\), hence maximized by a single edge of largest conductance; therefore \(\Acal_c(T)=\max_{e\in E(T)}c(e)\).
\end{proof}
\begin{figure}[H]
\centering
\begin{tabular}{@{}c c@{}}

\begin{minipage}{0.48\textwidth}
\centering
\begin{tabular}{@{}c c@{}}
\begin{tikzpicture}[
  scale=0.95,
  vtx/.style={circle, draw=black, fill=white, inner sep=1.3pt, line width=0.9pt},
  edg/.style={draw=black, line width=0.9pt},
  lab/.style={midway, fill=white, inner sep=1pt, font=\scriptsize}
]
\node[vtx] (p) at (0,0) {};
\node[vtx] (q) at (1.2,0.8) {};
\node[vtx] (r) at (1.2,-0.8) {};

\node[vtx] (s) at (2.5,1.2) {};
\node[vtx] (t) at (2.5,0.4) {};

\node[vtx] (u) at (2.5,-0.4) {};
\node[vtx] (v) at (2.5,-1.2) {};
\node[vtx] (w) at (2.5,-2.0) {};

\node[vtx] (x) at (3.7,-1.6) {};
\node[vtx] (y) at (3.7,-2.4) {};
\node[vtx] (z) at (4.9,-2.0) {};
\node[vtx] (m) at (4.9,-2.8) {};

\draw[edg] (p)--node[lab,above]{$2$}(q);
\draw[edg] (p)--node[lab,below]{$1$}(r);

\draw[edg] (q)--node[lab,above]{$3$}(s);
\draw[edg] (q)--node[lab,below]{$\tfrac{3}{2}$}(t);

\draw[edg] (r)--node[lab,above]{$\tfrac{7}{10}$}(u);
\draw[edg] (r)--node[lab,above, yshift=-5pt]{$\tfrac{11}{5}$}(v);
\draw[edg] (r)--node[lab,left]{$5$}(w);

\draw[edg] (w)--node[lab,above]{$\tfrac{11}{10}$}(x);
\draw[edg] (w)--node[lab,below]{$\tfrac{14}{5}$}(y);

\draw[edg] (y)--node[lab,above]{$\tfrac{19}{10}$}(z);
\draw[edg] (y)--node[lab,below]{$\tfrac{21}{5}$}(m);
\end{tikzpicture}
&
\(\displaystyle
\Acal_c(T_2)=\max_{e\in E(T_2)} c(e)=5.
\)
\end{tabular}

\small (b) Weighted tree \(T_2\) (\(|V|=12\))
\end{minipage}

\end{tabular}

\caption{A weighted tree with edge conductances labeled.}
\label{fig:two_weighted_trees}
\end{figure}
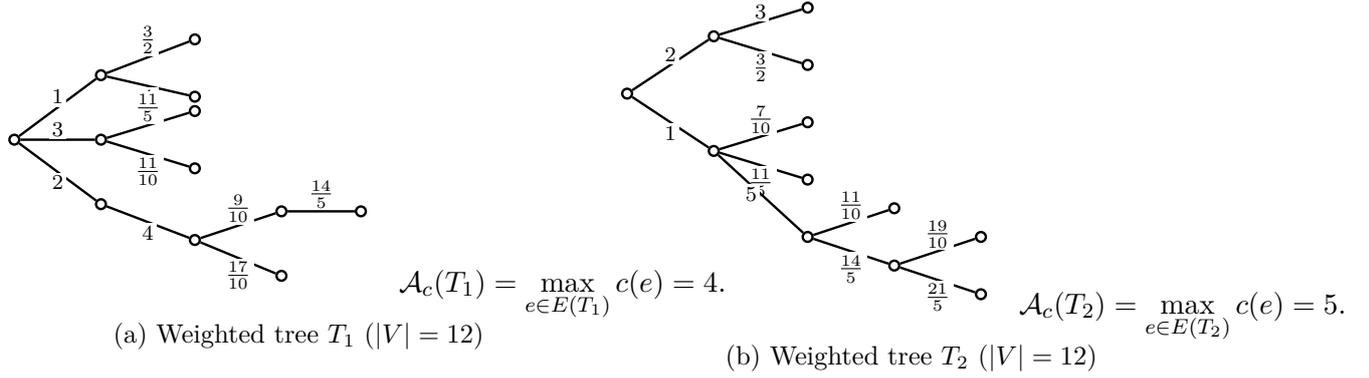

\begin{example}[Classical families, unit weights]\label{prop:examples}
If $c\equiv 1$ then:
\begin{enumerate}[label=(\alph*)]
\item If $G$ is a tree, then $\Acal_c(G)=1$.
\item For a cycle $C_n$ ($n\ge 3$), $\Acal_c(C_n)=\dfrac{n}{\,n-1\,}$.
\item For the complete graph $K_n$ ($n\ge 2$), $\Acal_c(K_n)=\dfrac{n}{2}$.
\end{enumerate}
\end{example}

\begin{figure}[H]

\centering
\begin{tikzpicture}[
  scale=0.95,
  vtx/.style={circle, draw=black, fill=white, inner sep=1.3pt, line width=0.9pt},
  edg/.style={draw=black, line width=0.7pt}
]
\foreach \i in {1,...,7}{
  \node[vtx] (u\i) at ({360/7*\i}:1.35) {};
}
\foreach \i in {1,...,7}{
  \foreach \j in {\i,...,7}{
    \ifnum\i<\j
      \draw[edg] (u\i)--(u\j);
    \fi
  }
}

\end{tikzpicture}

\vspace{2pt}
\[
c\equiv 1,\qquad
|V(K_7)|=7,\qquad |E(K_7)|=\binom{7}{2}=21,\qquad
D_1(K_7)=\frac{|E|}{|V|-1}=\frac{21}{6}=\frac{7}{2},
\qquad
\Acal_1(K_7)=\frac{7}{2}.
\]

\caption{Complete graph $K_7$ (unit conductances)}
\label{figk7}
\end{figure}
The evaluations from Example \ref{prop:examples} are well-known and follow immediately from the standard fractional arboricity characterization, that which omits the ceiling function from the classical Nash--Williams identity in \cite{nashwilliams1964}.

\section{Local Variants and Local Foster Constraints}\label{sec:local}

\subsection{Local weighted arboricity}

\begin{definition}[Local weighted arboricity]
For $v\in V(G)$, define
\[
\Acal^{\mathrm{loc}}_c(v)\ :=\ \max\Bigl\{ D_c(H) : H\subseteq G\ \text{connected},\ v\in V(H),\ |V(H)|\ge 2\Bigr\}.
\]
\end{definition}

\begin{proposition}[Local--global equivalence]\label{prop:local-global}
With $\Acal_c(G)$ as above,
\[
\Acal_c(G)\;=\;\max_{v\in V(G)} \Acal^{\mathrm{loc}}_c(v).
\]
\end{proposition}

\begin{proof}
Every connected $H$ contributing to $\Acal_c(G)$ contains some $v\in V(H)$ and hence contributes to $\Acal^{\mathrm{loc}}_c(v)$. Conversely, $\Acal^{\mathrm{loc}}_c(v)\le \Acal_c(G)$ holds by definition for each $v$. Taking maxima gives equality.
\end{proof}

\subsection{Effective resistance via energy dissipation}
Here we provide a brief but fully rigorous characterization of effective resistances relying heavily on standard results from \cite{spielman2010lect13}. Let \(r_e\) denote the ordinary edge resistance of \(e\), defined by
\[
    r_e := \frac{1}{w_e}.
\] Thus \(w_e\) is the conductance of \(e\), while \(r_e\) is its reciprocal resistance. Define the diagonal edge-resistance matrix
\[
    \bm{R} := \operatorname{diag}(r_e)_{e\in E(G)}.
\]
Then
\[
    \bm{W}=\bm{R}^{-1}
    =
    \operatorname{diag}(w_e)_{e\in E(G)}.
\]
For a matrix \(\bm{L}\), we write
\[
    \operatorname{span}(\bm{L})
    :=
    \operatorname{im}(\bm{L})
    =
    \{\bm{Lx}:\bm{x}\in\mathbb{R}^{V(G)}\},
\]
the column space of \(\bm{L}\). When \(G\) is connected, the weighted Laplacian satisfies
\[
    \ker(\bm{L})=\operatorname{span}\{\mathbf{1}\},
    \qquad
    \operatorname{span}(\bm{L})=\mathbf{1}^{\perp}
    =
    \left\{\bm{x}\in\mathbb{R}^{V(G)}:\mathbf{1}^{\top}\bm{x}=0\right\}.
\]
Consequently, the symmetric projection onto \(\operatorname{span}(\bm{L})\) is
\[
    \bm{\Pi}
    =
    \bm{I}-\frac{1}{|V(G)|}\mathbf{1}\mathbf{1}^{\top}.
\]
The Moore--Penrose pseudoinverse \(\bm{L}^{\dagger}\) satisfies
\[
    \bm{L}\bm{L}^{\dagger}
    =
    \bm{L}^{\dagger}\bm{L}
    =
    \bm{\Pi},
    \qquad
    \bm{L}\bm{L}^{\dagger}\bm{L}
    =
    \bm{L}.
\]

Finally, the effective resistance matrix of a connected weighted graph \(G\) is the vertex-indexed matrix
\[
    \mathcal{R}_G
    :=
    \bigl(\ReffG(a,b)\bigr)_{a,b\in V(G)},
\]
where
\[
    \ReffG(a,b)
    =
    (\mathbf{e}_a-\mathbf{e}_b)^{\top}
    \bm{L}^{\dagger}
    (\mathbf{e}_a-\mathbf{e}_b).
\]
Equivalently,
\[
    \bigl(\mathcal{R}_G\bigr)_{ab}
    =
    (\bm{L}^{\dagger})_{aa}
    +
    (\bm{L}^{\dagger})_{bb}
    -
    2(\bm{L}^{\dagger})_{ab}.
\]

With the pieces in place, we define effective resistance $\ReffG (a,b)$ as the resistance between some $a,b\in V(G)$ given by the whole network. This is useful for a standard but powerful characterization of effective resistance:

\begin{lemma}[Effective resistance]
    Let $\bm{i}$ be the one-unit electrical flow from $a$ to $b$ for $a,b\in V(G)$. Then $\ReffG (a,b) = \bm{\mathcal{E}(i)}$.
\end{lemma}
\begin{proof}
    Using $\bm{\iext}=\bm{Lv}$ and $\bm{v}=\bm{L^{\dagger}\iext} $, we have
    \begin{align*}
        \ReffG (a,b)=\bm{v}(a)-\bm{v}(b)=\bm{\iext^\top} \bm{v}=\bm{\iext^\top L^{\dagger}\iext} =\bm{v^\top LL^{\dagger}Lv}=\bm{v^\top Lv}=\bm{\mathcal{E}(i)}.
    \end{align*}
\end{proof}
Thus we have two invariant weighted metrics which, while admitting different interpretation of edge weights, embed information locally at each edge and can thus be compared piecewise analytically in a formal way; such is the focus of this section, so we can now define one of the main contributions of this paper.
\subsection{A local Foster inequality}

Let $\ReffX(e)$ denote the effective resistance between the endpoints of edge $e$ computed in the connected network $X$ with conductances. We use:
\begin{itemize}
\item (\emph{Foster’s First Theorem} \cite{foster1949}) For $X=(H,c|_{E(H)})$, $\sum_{e\in E(H)} c(e)\,\ReffX(e)=|V(H)|-1$.
\item (\emph{Rayleigh monotonicity}) Deleting edges or reducing conductances does not decrease effective resistance between any two vertices.
\end{itemize}

\begin{theorem}[Local Foster inequality]\label{thm:local-foster}
Let $H\subseteq G$ be any connected subgraph with inherited conductances. Then
\[
\sum_{e\in E(H)}c(e)\,\ReffG(e)\ \le\ |V(H)|-1.
\]
\end{theorem}

\begin{proof}
Let $X:=(H,c|_{E(H)})$. For each $e\in E(H)$, Rayleigh monotonicity yields $\ReffG(e)\le \ReffX(e)$. Multiplying by $c(e)$ and summing over $e\in E(H)$ gives
\[
\sum_{e\in E(H)}c(e)\,\ReffG(e)\ \le\ \sum_{e\in E(H)} c(e)\,\ReffX(e)\ =\ |V(H)|-1,
\]
using Foster's First Theorem on $X$.
\end{proof}

\begin{proposition}[An explicit upper bound]\label{prop:cs-upper}
For any connected $H\subseteq G$ with $|V(H)|\ge 2$,
\[
D_c(H)\;\le\;\sqrt{\,\frac{\displaystyle\sum_{e\in E(H)} \dfrac{c(e)}{\ReffG(e)}}{\,|V(H)|-1\,}}\,.
\]
Consequently,
\[
\Acal_c(G)\ \le\ \max_{\substack{H\subseteq G\\ H\ \mathrm{connected}}}
\sqrt{\,\frac{\displaystyle\sum_{e\in E(H)} \dfrac{c(e)}{\ \ReffG(e)}}{\,|V(H)|-1\,}}\,.
\]
\end{proposition}

\begin{proof}
Let $a_e:=\sqrt{c(e)\ReffG(e)}$ and $b_e:=\sqrt{c(e)/\ReffG(e)}$. Then \begin{align*}\sum_{e\in E(H)}c(e)=\sum a_e b_e\le \big(\sum a_e^2\big)^{1/2}\big(\sum b_e^2\big)^{1/2}\text{ (by Cauchy--Schwarz) }\\=\sqrt{\big(\sum c(e)\ReffG(e){\big)\big(\sum c(e)/\ReffG(e)}\big)}.\end{align*} By Theorem~\ref{thm:local-foster}, $\sum_{e\in E(H)} c(e)\,\ReffG(e) \le |V(H)|-1$, hence
\[
\Bigl(\sum_{e\in E(H)} c(e)\Bigr)^2
\le
(|V(H)|-1)\Bigl(\sum_{e\in E(H)} \frac{c(e)}{\ReffG(e)}\Bigr).
\]
Dividing by $|V(H)|-1$,
\[
\frac{\sum_{e\in E(H)} c(e)}{|V(H)|-1} = D_c (H)
\le
\sqrt{\frac{\displaystyle\sum_{e\in E(H)} \frac{c(e)}{\ReffG(e)}}{|V(H)|-1}}.
\]
Thus take maxima over subgraphs to yield the result.
\end{proof}
\begin{figure}[H]
\centering
\begin{tikzpicture}[
    vertex/.style={circle, draw=black!80, fill=blue!20, thick, minimum size=8mm},
    edge/.style={draw=black!70, thick},
    label/.style={font=\small, fill=white, inner sep=2pt},
    >=Stealth
]

\node[vertex] (v1) at (0,0) {$v_1$};
\node[vertex] (v2) at (2.5,0) {$v_2$};
\node[vertex] (v3) at (1.25,2) {$v_3$};
\node[vertex] (v4) at (4,1.5) {$v_4$};

\draw[edge] (v1) -- node[label, below] {$c=2$} (v2);
\draw[edge] (v1) -- node[label, left] {$c=1$} (v3);
\draw[edge] (v2) -- node[label, right] {$c=3$} (v3);
\draw[edge] (v2) -- node[label, below right] {$c=1$} (v4);
\draw[edge] (v3) -- node[label, above] {$c=2$} (v4);

\begin{scope}[on background layer]
\node[draw=red!60, thick, dashed, rounded corners=8pt, fit={(v1) (v2) (v3)}, 
      inner sep=8pt, fill=red!5] (subgraph) {};
\end{scope}

\node[below=0.2cm of subgraph, text=red!70!black, font=\small\itshape] {Subgraph $H$};

\node[below=1.5cm of v2, align=left, font=\footnotesize, draw=black!50, 
      fill=yellow!10, rounded corners=3pt, inner sep=6pt] (calc) {
    For $H = \{v_1, v_2, v_3\}$: \\[2pt]
    $|V(H)| = 3$, \quad $|E(H)| = 3$ \\[2pt]
    $\sum_{e \in E(H)} c(e) = 2 + 1 + 3 = 6$ \\[2pt]
    $D_c(H) = \frac{6}{3-1} = 3$
};

\node[right=1.5cm of calc, align=left, font=\footnotesize, draw=black!50,
      fill=green!10, rounded corners=3pt, inner sep=6pt] (inequality) {
    \textbf{Local Foster:} \\[2pt]
    $\displaystyle\sum_{e \in E(H)} c(e) \ReffG(e) \leq |V(H)| - 1$ \\[4pt]
    Implies: \\[2pt]
    $\displaystyle D_c(H) \leq \sqrt{\frac{\displaystyle\sum_{e \in E(H)} \frac{c(e)}{\ReffG(e)}}{|V(H)|-1}}$
};

\end{tikzpicture}
\caption{A small weighted network illustrating the Local Foster inequality from Theorem~\ref{thm:local-foster}. 
The highlighted subgraph $H$ contains vertices $\{v_1, v_2, v_3\}$ with three edges carrying 
conductances $c = 2, 1, 3$. The inequality $\sum_{e \in E(H)} c(e) \ReffG(e) \leq |V(H)| - 1$ 
provides an upper bound on the sum of conductance-weighted effective resistances, which in turn yields 
the explicit upper bound for the weighted density $D_c(H)$.}
\label{fig:foster_rayleigh_inequality}
\end{figure}
\begin{proposition}[Immediate local lower bounds]\label{prop:local-lower}
For any connected $H\subseteq G$ with $|V(H)|\ge 2$,
\[
\Acal_c(G)\ \ge\ D_c(H)\ =\ \frac{\sum_{e\in E(H)} c(e)}{|V(H)|-1}
\ \ge\ \max_{e\in E(H)} c(e).
\]
\end{proposition}

\begin{proof}
The first inequality holds since $\Acal_c(G)$ is the maximum over connected $H$. The second follows from the feasible two-vertex subgraph containing a single edge $e$.
\end{proof}

\begin{example}[Separation from totals/degree]\label{prop:sep}
With unit weights $c\equiv 1$, the triangle $K_3$ and the $3$-edge star $K_{1,3}$ both satisfy $C(G)=3$ and $\Delta_1(G)=2$, yet
\[
\Acal_1(K_3)=\frac{3}{2}\qquad\text{and}\qquad \Acal_1(K_{1,3})=1.
\]
\end{example}

\begin{figure}[h]
\centering
\begin{minipage}{0.46\textwidth}
\centering
\begin{tikzpicture}[scale=1.0,
  vtx/.style={circle, draw=black, fill=white, inner sep=1.3pt, line width=0.9pt},
  edg/.style={draw=black, line width=0.9pt}
]
\node[vtx] (a) at (0,0) {};
\node[vtx] (b) at (1.4,0) {};
\node[vtx] (c) at (0.7,1.2) {};
\draw[edg] (a)--(b)--(c)--(a);
\end{tikzpicture}

\small $K_3$
\end{minipage}
\hfill
\begin{minipage}{0.46\textwidth}
\centering
\begin{tikzpicture}[scale=1.0,
  vtx/.style={circle, draw=black, fill=white, inner sep=1.3pt, line width=0.9pt},
  edg/.style={draw=black, line width=0.9pt}
]
\node[vtx] (o) at (0,0) {};
\node[vtx] (u) at (0,1.2) {};
\node[vtx] (l) at (-1.1,-0.7) {};
\node[vtx] (r) at (1.1,-0.7) {};
\draw[edg] (o)--(u);
\draw[edg] (o)--(l);
\draw[edg] (o)--(r);
\end{tikzpicture}

\small $K_{1,3}$
\end{minipage}
\caption{Two graphs with the same totals ($C(G)=3$) and maximum degree ($\Delta_1(G)=2$), but different $\mathcal{A}_1$.}
\end{figure}

These observations follow immediately from the fact that $\Acal_1(G)$ is exactly fractional arboricity. Further results on heterogeneous $c$ are discussed in Section~\ref{sec:computation}.
\section{Algebraic structure}\label{sec:algebra}

In addition to its analytic and extremal properties, the weighted arboricity
functional $\Acal_c(G)$ interacts in a clean way with natural constructions on weighted
graphs. In this section we investigate one such feature:
 $\Acal_c(G)$ behaves predictably under edge-disjoint unions. At the level of
isomorphism classes of weighted graphs, disjoint union induces a commutative
monoid structure, and $\Acal_c(G)$ is idempotent with respect to this operation
(Section~\ref{subsec:disjoint-unions}). These observations justify
viewing $\Acal_c(G)$ like a max invariant in structure with a simple algebraic profile.

\subsection{Disjoint unions and max structure}
\label{subsec:disjoint-unions}
The behavior of $\Acal_c(G)$ under weighted embeddings reflects its
monotonicity with respect to enlarging graphs and increasing edge weights.
A complementary structural feature appears when we combine graphs via
edge-disjoint union. In that setting, $\Acal_c(G)$ is determined simply by the
maximum of the component arboricities, and this induces a commutative
idempotent monoid structure at the level of isomorphism classes.
We record this next.
\begin{definition}[Edge-disjoint union of weighted graphs]
Let $(G_1,c_1)$ and $(G_2,c_2)$ be finite weighted graphs. Assume that
their vertex sets are disjoint and hence there are no edges between
$V(G_1)$ and $V(G_2)$.

The \emph{edge-disjoint union} of $(G_1,c_1)$ and $(G_2,c_2)$ is the
weighted graph $(G,c)$ defined by
\[
G := G_1 \sqcup G_2,\qquad
V(G) := V(G_1)\cup V(G_2),\quad
E(G) := E(G_1)\cup E(G_2),
\]
and
\[
c(e) :=
\begin{cases}
c_1(e), & e\in E(G_1),\\[2pt]
c_2(e), & e\in E(G_2).
\end{cases}
\]
By construction, $G$ has exactly two connected components $G_1$ and
$G_2$.
\end{definition}

\begin{proposition}[Weighted arboricity of an edge-disjoint union]
\label{prop:disjoint-union}
Let $(G_1,c_1)$ and $(G_2,c_2)$ be finite weighted graphs, and let
$(G,c) := (G_1,c_1)\sqcup (G_2,c_2)$ be their edge-disjoint union as
above. Then
\[
\Acal_c(G)
\;=\;
\max\bigl\{\Acal_{c_1}(G_1),\,\Acal_{c_2}(G_2)\bigr\}.
\]

More generally, if $(G,c)$ is the edge-disjoint union of finitely many
weighted graphs $(G_i,c_i)$, $1\le i\le k$, then
\[
\Acal_c(G)
\;=\;
\max_{1\le i\le k} \Acal_{c_i}(G_i).
\]
\end{proposition}
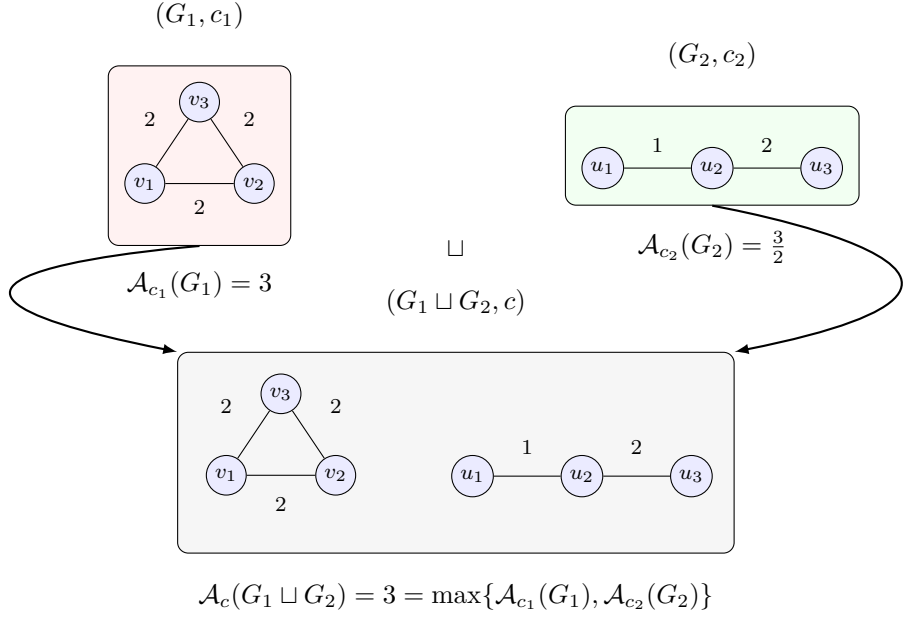
\begin{figure}[H]
\centering
\begin{tikzpicture}[>=latex, node distance=3.6cm, every node/.style={font=\small}]

\node[draw, rounded corners, fill=red!5, inner sep=6pt] (G1box) {
  \begin{tikzpicture}[scale=0.9, every node/.style={font=\scriptsize}]
    \node[circle,draw,fill=blue!8,inner sep=2pt] (v1) at (0,0) {$v_1$};
    \node[circle,draw,fill=blue!8,inner sep=2pt] (v2) at (1.6,0) {$v_2$};
    \node[circle,draw,fill=blue!8,inner sep=2pt] (v3) at (0.8,1.2) {$v_3$};
    \draw (v1) -- node[below] {$2$} (v2);
    \draw (v1) -- node[above left] {$2$} (v3);
    \draw (v2) -- node[above right] {$2$} (v3);
  \end{tikzpicture}
};
\node[above=0.35cm of G1box] {$ (G_1,c_1)$};
\node[below=0.2cm of G1box] (A1label) {$ \Acal_{c_1}(G_1) = 3 $};

\node[draw, rounded corners, fill=green!5, inner sep=6pt, right=of G1box] (G2box) {
  \begin{tikzpicture}[scale=0.9, every node/.style={font=\scriptsize}]
    \node[circle,draw,fill=blue!8,inner sep=2pt] (u1) at (0,0) {$u_1$};
    \node[circle,draw,fill=blue!8,inner sep=2pt] (u2) at (1.6,0) {$u_2$};
    \node[circle,draw,fill=blue!8,inner sep=2pt] (u3) at (3.2,0) {$u_3$};
    \draw (u1) -- node[above] {$1$} (u2);
    \draw (u2) -- node[above] {$2$} (u3);
  \end{tikzpicture}
};
\node[above=0.35cm of G2box] {$ (G_2,c_2)$};
\node[below=0.2cm of G2box] (A2label) {$ \Acal_{c_2}(G_2) = \tfrac{3}{2} $};

\node[draw, rounded corners, fill=gray!7, inner sep=8pt,
      below=2.6cm of $(G1box)!0.5!(G2box)$] (Gbox) {
  \begin{tikzpicture}[scale=0.9, every node/.style={font=\scriptsize}]
    \node[circle,draw,fill=blue!8,inner sep=2pt] (v1) at (0,0) {$v_1$};
    \node[circle,draw,fill=blue!8,inner sep=2pt] (v2) at (1.6,0) {$v_2$};
    \node[circle,draw,fill=blue!8,inner sep=2pt] (v3) at (0.8,1.2) {$v_3$};
    \draw (v1) -- node[below] {$2$} (v2);
    \draw (v1) -- node[above left] {$2$} (v3);
    \draw (v2) -- node[above right] {$2$} (v3);

    \node[circle,draw,fill=blue!8,inner sep=2pt] (u1) at (3.6,0) {$u_1$};
    \node[circle,draw,fill=blue!8,inner sep=2pt] (u2) at (5.2,0) {$u_2$};
    \node[circle,draw,fill=blue!8,inner sep=2pt] (u3) at (6.8,0) {$u_3$};
    \draw (u1) -- node[above] {$1$} (u2);
    \draw (u2) -- node[above] {$2$} (u3);
  \end{tikzpicture}
};
\node[above=0.35cm of Gbox] {$ (G_1 \sqcup G_2, c)$};
\node[below=0.25cm of Gbox] (Alabel)
  {$ \Acal_c(G_1 \sqcup G_2)
      = 3
      = \max\{\Acal_{c_1}(G_1), \Acal_{c_2}(G_2)\}$};

\draw[->, thick]
  (G1box.south) .. controls +(0,-0) and +(-5.0,1.0) .. (Gbox.north west);
\draw[->, thick]
  (G2box.south) .. controls +(0,0) and +(5.0,1.0) .. (Gbox.north east);
\node at ($(G1box.south)!0.5!(G2box.south) + (0,-0.3)$) {$\sqcup$};

\end{tikzpicture}
\caption{Disjoint union as a monoid: $\Acal_c$ is a homomorphism
$(\mathcal{G},\sqcup) \to (\mathbb{R}_{\ge 0},\max)$.}
\end{figure}
\FloatBarrier

\begin{proof}
We first prove the two-component case.

Let $(G,c)=(G_1,c_1)\sqcup(G_2,c_2)$.
By definition,
\[
\Acal_c(G)
=
\max\Bigl\{ D_c(H) \,:\, H\subseteq G\ \text{connected},\ |V(H)|\ge 2\Bigr\}.
\]

\smallskip\noindent
\emph{Step 1: Any connected subgraph lies in a single component.}
Let $H\subseteq G$ be a connected subgraph with $|V(H)|\ge 2$. Since
there are no edges between $V(G_1)$ and $V(G_2)$, any path in $G$
lies entirely inside $G_1$ or entirely inside $G_2$.  
If $V(H)$ met both $V(G_1)$ and $V(G_2)$, then there would be no path
inside $H$ joining a vertex from $V(G_1)\cap V(H)$ to a vertex from
$V(G_2)\cap V(H)$, contradicting the connectedness of $H$.
Therefore
\[
H\subseteq G_1
\quad\text{or}\quad
H\subseteq G_2.
\]

\smallskip\noindent
\emph{Step 2: Densities agree with the component densities.}
Suppose $H\subseteq G_i$ for some $i\in\{1,2\}$. Then, by the
definition of $c$,
\[
\sum_{e\in E(H)} c(e)
=
\sum_{e\in E(H)} c_i(e).
\]
Since $V(H)$ is the same regardless of whether we regard $H$ as a
subgraph of $G$ or of $G_i$, it follows that
\[
D_c(H)
=
\frac{\sum_{e\in E(H)} c(e)}{|V(H)|-1}
=
\frac{\sum_{e\in E(H)} c_i(e)}{|V(H)|-1}
=
D_{c_i}(H).
\]

\smallskip\noindent
\emph{Step 3: Maximization splits over components.}
Using Step~1 and Step~2, we can rewrite the definition of $\Acal_c(G)$ as
\begin{align*}
\Acal_c(G)
&= \max\Bigl\{ D_c(H) \,:\, H\subseteq G\ \text{connected},\ |V(H)|\ge 2\Bigr\}\\
&= \small \max\Bigl(
   \max\{D_c(H) : H\subseteq G_1\ \text{connected},\ |V(H)|\ge 2\},\,
   \max\{D_c(H) : H\subseteq G_2\ \text{connected},\ |V(H)|\ge 2\}
   \Bigr)\\
&= \max\bigl\{\Acal_{c_1}(G_1),\,\Acal_{c_2}(G_2)\bigr\},
\end{align*}
which is the desired identity.

\smallskip
For a finite disjoint union of $k\ge 2$ components, we argue by
induction on $k$. The case $k=2$ is exactly the argument above.
Assume the statement holds for $k-1$ components. Write
\[
(G,c) = \Bigl(\bigsqcup_{i=1}^{k-1}(G_i,c_i)\Bigr)\sqcup(G_k,c_k)
       =: (H,d)\sqcup(G_k,c_k).
\]
By the two-component case,
\[
\Acal_c(G)
=
\max\bigl\{\Acal_d(H),\,\Acal_{c_k}(G_k)\bigr\}.
\]
By the induction hypothesis applied to $(H,d)$,
\[
\Acal_d(H) = \max_{1\le i\le k-1} \Acal_{c_i}(G_i).
\]
Combining these gives
\[
\Acal_c(G)
=
\max\Bigl(\max_{1\le i\le k-1} \Acal_{c_i}(G_i),\,\Acal_{c_k}(G_k)\Bigr)
=
\max_{1\le i\le k} \Acal_{c_i}(G_i),
\]
completing the induction and the proof.
\end{proof}

\begin{corollary}[Idempotent behavior under disjoint union]
Let $\mathcal{G}$ denote the set of isomorphism classes of finite
weighted graphs, and let $\sqcup$ denote edge-disjoint union on
representatives. Then:

\begin{enumerate}
\item $(\mathcal{G},\sqcup)$ is a commutative monoid with identity
element given by the edgeless graph (with arbitrary vertex set and
zero weights).
\item For any $[G,c]\in\mathcal{G}$,
\[
  \Acal_c(G\sqcup G)
=
\max\{\Acal_c(G),\Acal_c(G)\}
=
\Acal_c(G),
\]
so $\Acal_c(G)$ is idempotent with respect to disjoint union at the level of
its values.
\end{enumerate}
\end{corollary}

\section{Computation}\label{sec:computation}

This section gives a small computational exhibit for the numerical procedure used to generate the overlay figures comparing the baseline density \(A_1(G)\) against the upper bound from Section~\ref{sec:local}, \(\mathrm{CSBound}(H{=}G)\), in both the unit--conductance and random $c$ cases. All plots in this section were produced by the script \texttt{scripts/make\_figs.py}, which outputs both an overlay plot of \(A_c(G)\) and \(\mathrm{CSBound}(H{=}G)\) versus \(v=|V|\) for some common graph families and a tightness--ratio plot \(\mathrm{CSBound}(H{=}G)/A_c(G)\). While the exhibits may be very useful for future analysis, their inclusion in this paper will be limited to one graph family for brevity. We restrict attention here to the hypercube family $Q_d$, which is highly structured (edge–transitive) and thus admits a clean, representative testbed for the bound.

\subsection{Quantities plotted}

For a graph \(G=(V,E)\) with \(|V|=v\ge 2\), the script plots the unit--weight functional
\[
A_1(G)\;=\;\max\Bigl\{\frac{|E(H)|}{|V(H)|-1}:\ H \subseteq G\ \text{is connected}\Bigr\}.
\]
For very small graphs (\(v\le 12\) by default) this is computed by brute force enumeration over all connected induced vertex subsets. For larger graphs, the script returns \(A_1(G)\) only in families where the exact value is known a priori without further optimization: for trees $T$, including paths and stars, \(A_1(T)=1\); for complete graphs, cycles, complete bipartite graphs, and hypercubes, \(A_1(G)=D_1(G)=|E|/(|V|-1)\). 
The explicit upper bound evaluated at \(H=G\) in the unit conductance case comes from Proposition~\ref{prop:cs-upper}:
\[
\mathrm{CSBound}(H{=}G)
\;:=\;
D_c(H=G)\;\le\;\sqrt{\,\frac{\displaystyle\sum_{e\in E(H)} \dfrac{c(e)}{\ReffG(e)}}{\,|V(H)|-1\,}}\,.
\]
where $\ReffG(e)$ is the effective resistance between the endpoints of \(e\) in the ambient graph \(G\).
In the overlay figures we plot the pair \(\bigl(A_1(G),\mathrm{CSBound}(H{=}G)\bigr)\) against \(v=|V|\), and in the tightness figures we plot the ratio \(\mathrm{CSBound}(H{=}G)/A_1(G)\).

\subsection{Computation on the hypercube family}\label{sec:computation_hypercube}

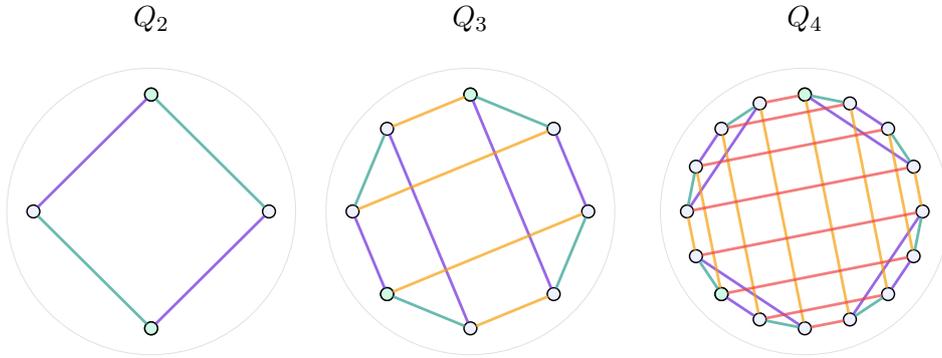
\begin{figure}[H]
\centering
\begin{tikzpicture}[x=1cm,y=1cm, line cap=round, line join=round]

\definecolor{bitA}{HTML}{2A9D8F} 
\definecolor{bitB}{HTML}{6D28D9} 
\definecolor{bitC}{HTML}{F59E0B} 
\definecolor{bitD}{HTML}{EF4444} 

\definecolor{nodeMid}{HTML}{EEF2FF} 
\definecolor{nodeHi}{HTML}{D1FAE5}  

\tikzset{
  vtx/.style={circle, draw=black, line width=.55pt, fill=nodeMid, inner sep=1.7pt},
  vtxHi/.style={vtx, fill=nodeHi},
  edgeA/.style={line width=1.05pt, draw=bitA, opacity=.70},
  edgeB/.style={line width=1.05pt, draw=bitB, opacity=.70},
  edgeC/.style={line width=1.05pt, draw=bitC, opacity=.70},
  edgeD/.style={line width=1.05pt, draw=bitD, opacity=.70},
  title/.style={font=\bfseries}
}

\begin{scope}[shift={(0,0)}]
  \node[title] at (0,2.55) {$Q_2$};

  \def\R{1.55}
  \foreach[count=\i from 0] \lab in {00,01,11,10}{
    \pgfmathsetmacro\ang{90 - 360*\i/4}
    \coordinate (\lab) at ({\R*cos(\ang)},{\R*sin(\ang)});
    \ifx\lab\empty\else\fi
    \node[vtx] (v\lab) at (\lab) {};
  }
  \node[vtxHi] (v00) at (00) {};
  \node[vtxHi] (v11) at (11) {};

  \draw[edgeA] (v00)--(v01);
  \draw[edgeA] (v10)--(v11);
  \draw[edgeB] (v00)--(v10);
  \draw[edgeB] (v01)--(v11);

  \draw[opacity=.12] (0,0) circle (\R+0.35);
\end{scope}

\begin{scope}[shift={(4.2,0)}]
  \node[title] at (0,2.55) {$Q_3$};

  \def\R{1.55}
  \foreach[count=\i from 0] \lab in {000,001,011,010,110,111,101,100}{
    \pgfmathsetmacro\ang{90 - 360*\i/8}
    \coordinate (\lab) at ({\R*cos(\ang)},{\R*sin(\ang)});
    \node[vtx] (v\lab) at (\lab) {};
  }
  \node[vtxHi] (v000) at (000) {};
  \node[vtxHi] (v111) at (111) {};

  \draw[edgeA] (v000)--(v001);
  \draw[edgeA] (v010)--(v011);
  \draw[edgeA] (v100)--(v101);
  \draw[edgeA] (v110)--(v111);

  \draw[edgeB] (v000)--(v010);
  \draw[edgeB] (v001)--(v011);
  \draw[edgeB] (v100)--(v110);
  \draw[edgeB] (v101)--(v111);

  \draw[edgeC] (v000)--(v100);
  \draw[edgeC] (v001)--(v101);
  \draw[edgeC] (v010)--(v110);
  \draw[edgeC] (v011)--(v111);

  \draw[opacity=.12] (0,0) circle (\R+0.35);
\end{scope}

\begin{scope}[shift={(8.6,0)}]
  \node[title] at (0,2.55) {$Q_4$};

  \def\R{1.55}
  \foreach[count=\i from 0] \lab in {%
    0000,0001,0011,0010,0110,0111,0101,0100,%
    1100,1101,1111,1110,1010,1011,1001,1000%
  }{
    \pgfmathsetmacro\ang{90 - 360*\i/16}
    \coordinate (\lab) at ({\R*cos(\ang)},{\R*sin(\ang)});
    \node[vtx] (v\lab) at (\lab) {};
  }
  \node[vtxHi] (v0000) at (0000) {};
  \node[vtxHi] (v1111) at (1111) {};

  \draw[edgeA] (v0000)--(v0001);
  \draw[edgeA] (v0010)--(v0011);
  \draw[edgeA] (v0100)--(v0101);
  \draw[edgeA] (v0110)--(v0111);
  \draw[edgeA] (v1000)--(v1001);
  \draw[edgeA] (v1010)--(v1011);
  \draw[edgeA] (v1100)--(v1101);
  \draw[edgeA] (v1110)--(v1111);

  \draw[edgeB] (v0000)--(v0010);
  \draw[edgeB] (v0001)--(v0011);
  \draw[edgeB] (v0100)--(v0110);
  \draw[edgeB] (v0101)--(v0111);
  \draw[edgeB] (v1000)--(v1010);
  \draw[edgeB] (v1001)--(v1011);
  \draw[edgeB] (v1100)--(v1110);
  \draw[edgeB] (v1101)--(v1111);

  \draw[edgeC] (v0000)--(v0100);
  \draw[edgeC] (v0001)--(v0101);
  \draw[edgeC] (v0010)--(v0110);
  \draw[edgeC] (v0011)--(v0111);
  \draw[edgeC] (v1000)--(v1100);
  \draw[edgeC] (v1001)--(v1101);
  \draw[edgeC] (v1010)--(v1110);
  \draw[edgeC] (v1011)--(v1111);

  \draw[edgeD] (v0000)--(v1000);
  \draw[edgeD] (v0001)--(v1001);
  \draw[edgeD] (v0010)--(v1010);
  \draw[edgeD] (v0011)--(v1011);
  \draw[edgeD] (v0100)--(v1100);
  \draw[edgeD] (v0101)--(v1101);
  \draw[edgeD] (v0110)--(v1110);
  \draw[edgeD] (v0111)--(v1111);

  \draw[opacity=.12] (0,0) circle (\R+0.35);

\end{scope}

\end{tikzpicture}
\caption{Hypercube family $Q_d$ for $d=2,3,4$}
\end{figure}
We illustrate the computation of the upper bound $\mathrm{CSBound}(H{=}G)$ on the hypercube family $Q_d$.
Since $Q_d$ is edge-transitive, all edges lie in one automorphism orbit; hence the
ambient effective resistance $\ReffQd(e)$ is the same for every edge
$e\in E(Q_d)$ \cite{anari-oveisgharan2015, klein-randic1993}.

Fix any edge $e_0=\{u,v\}$ and write
\[
R_d \;:=\; \ReffQd(e_0) \;=\; \ReffQd(e)
\  \text{for all } e\in E(Q_d).
\]
Hence for an arbitrary conductance assignment $c:E(Q_d)\to(0,\infty)$,
\[
\sum_{e\in E(Q_d)} \frac{c(e)}{\ReffQd(e)}
\;=\;
\frac{1}{R_d}\sum_{e\in E(Q_d)} c(e).
\]

To compute $\ReffQd(e)$, the script follows from \cite{ghosh2008} and uses a grounded Laplacian solve: relabel vertices as
$\{0,1,\dots,|V|-1\}$ with ground node $0$, then form the reduced (grounded) Laplacian by deleting the row and column of the
Laplacian corresponding to node $0$; in other words, let $L_{\mathrm{gr}}$ denote the principal submatrix of $L$ obtained by deleting the row and column corresponding to the ground vertex $0$. Then for $u,v\in V(G)$, the endpoints of $e\in E(G)$, solve $L_{\mathrm{gr}}x=b$ with $b$ equal to $+1$ at $u$ and $-1$ at $v$
(with the grounded coordinate omitted). The effective resistance is then the resulting voltage drop
$R_{\mathrm{eff}(e)}^{(Q_d)}=|x_u-x_v|$ \cite{ghosh2008}.

\subsection{Random conductance sampling}\label{sec:rand_c_method}

To test robustness under heterogeneous weights, we generate i.i.d.\@ random conductances on the hypercube edges from a two-point law supported on $\{1,\lambda\}$. Fix parameters $(\lambda,p)$ and, independently for each edge $e\in E(Q_d)$, sample
\[
c(e)=
\begin{cases}
\lambda, & \text{with probability } p,\\
1, & \text{with probability } 1-p.
\end{cases}
\]
For each dimension $d\in\{d_{\min},\dots,d_{\max}\}$ we draw \texttt{reps} independent conductance assignments and compute, for each draw, the overlay quantities $D_c(Q_d)$ and $\mathrm{CSBound}(H{=}Q_d)$, as well as the tightness ratio $\mathrm{CSBound}(H{=}Q_d)/D_c(Q_d)$.
We summarize trial-to-trial variation in the overlay plots via shaded bands, and we report per-$d$ median ratios in the tightness figures.

As a representative run, the command
\begin{verbatim}
 --dist two_point --lambda 10 --p 0.5 --d_min 4 --d_max 10 --reps 25
\end{verbatim}
produces median tightness ratios decreasing with $d$, with total runtime on the order of one minute
for the full sweep.
Unless stated otherwise, we use $p=0.5$ and repeat the same experiment for $\lambda\in\{5,10,50\}$.
\subsection{Plots}
\begin{figure}[H]
    \centering
    \includegraphics[width=0.85\linewidth]{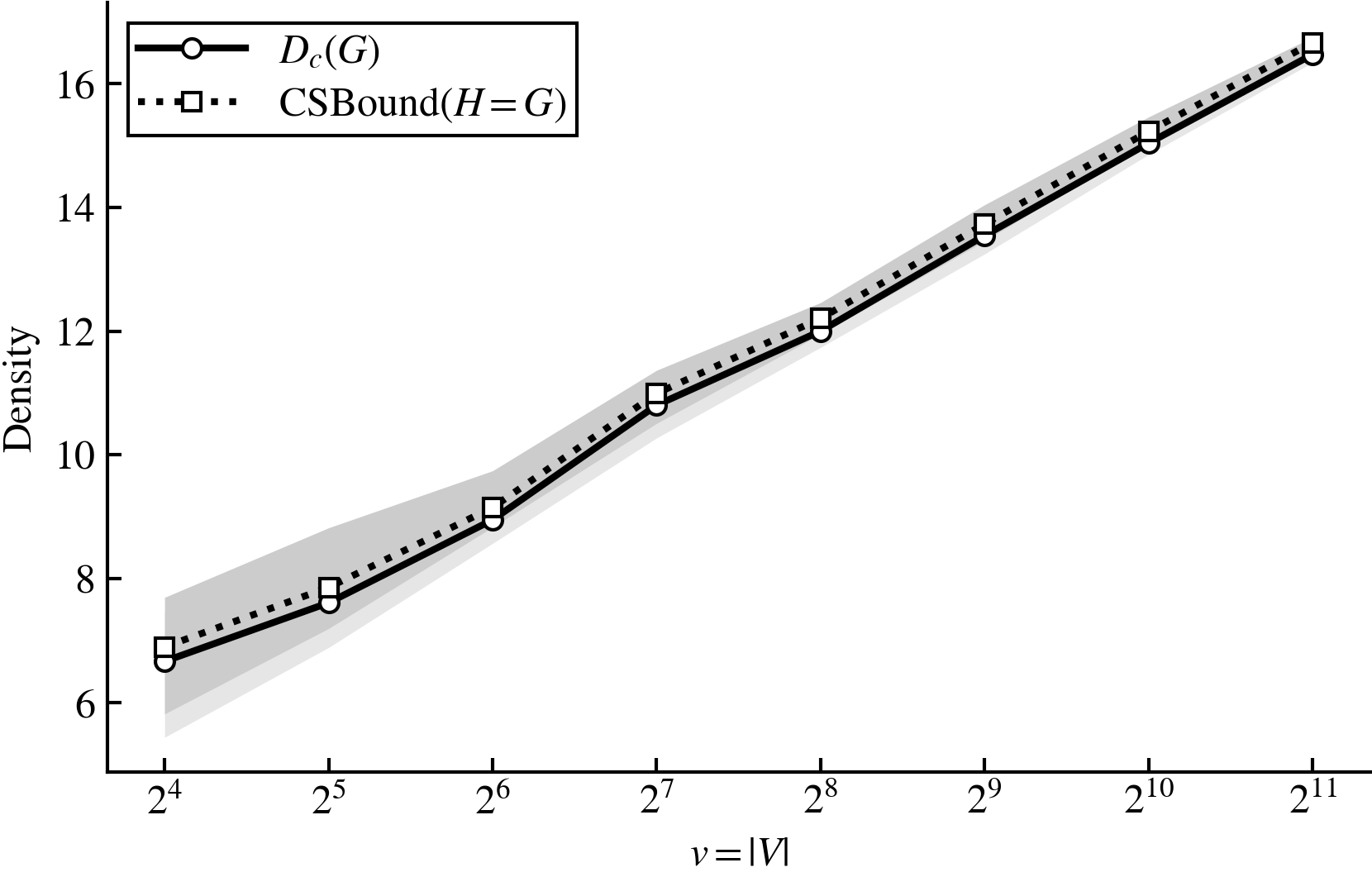}
    \caption{Densities with heterogeneous $c$ and $\lambda\in [1, 5]$}
    \label{fig:compute1}
\end{figure}

\begin{figure}[H]
    \centering
    \includegraphics[width=0.85\linewidth]{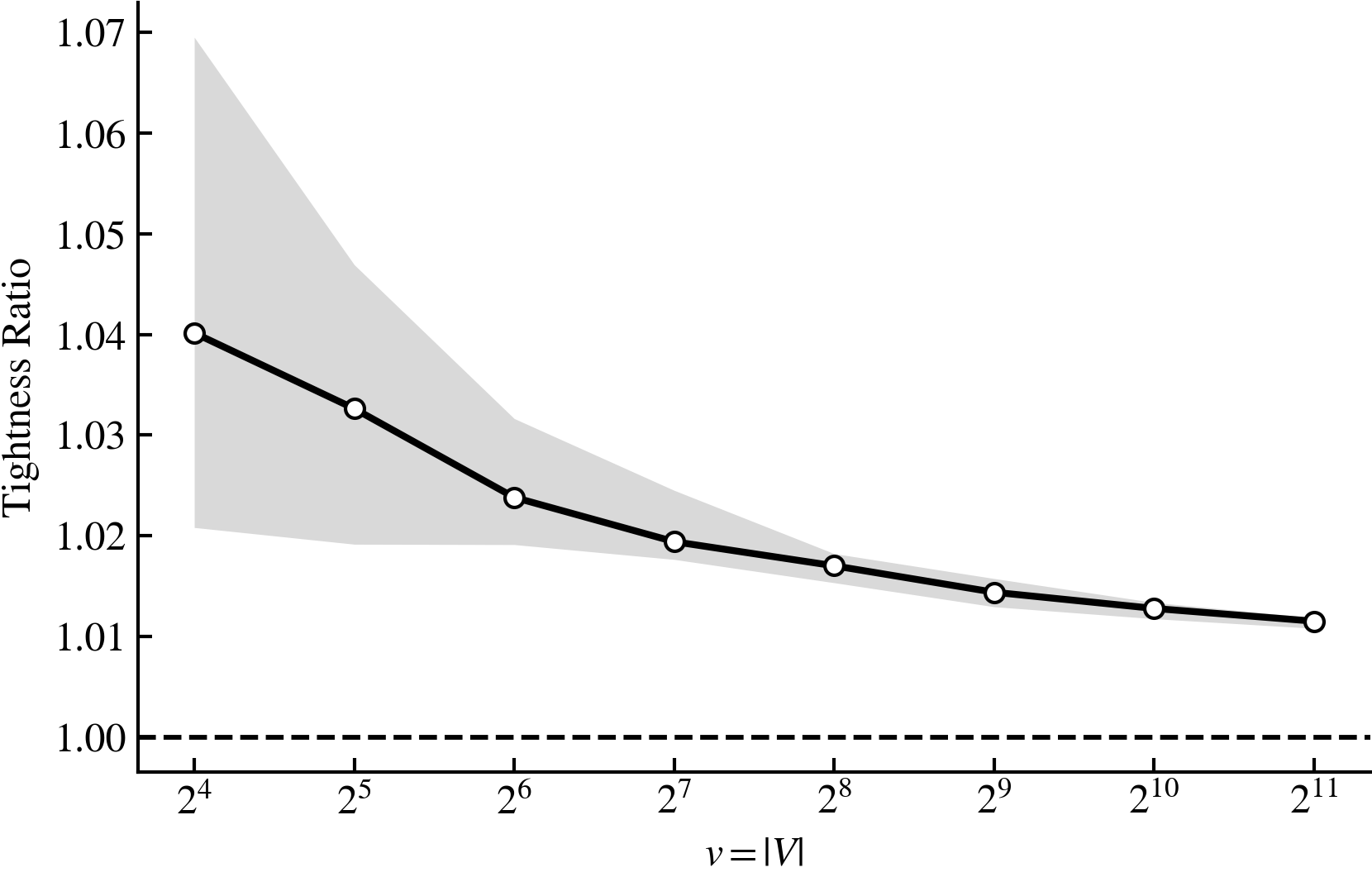}
    \caption{Tightness with heterogeneous $c$ and $\lambda\in [1, 5]$}
    \label{fig:compute1}
\end{figure}
\begin{remark}
    Pleasantly, our properties from Section~\ref{sec:basic} show up in the asymptotic behavior of the $D_c(G)$ functional (prior to maximization over connected subgraphs), such as monotonicity. Additionally, the upper bound is tight at large $v$ (although these are all low-dimensional hypercube families.)
\end{remark}

\begin{figure}[H]
    \centering
    \includegraphics[width=0.85\linewidth]{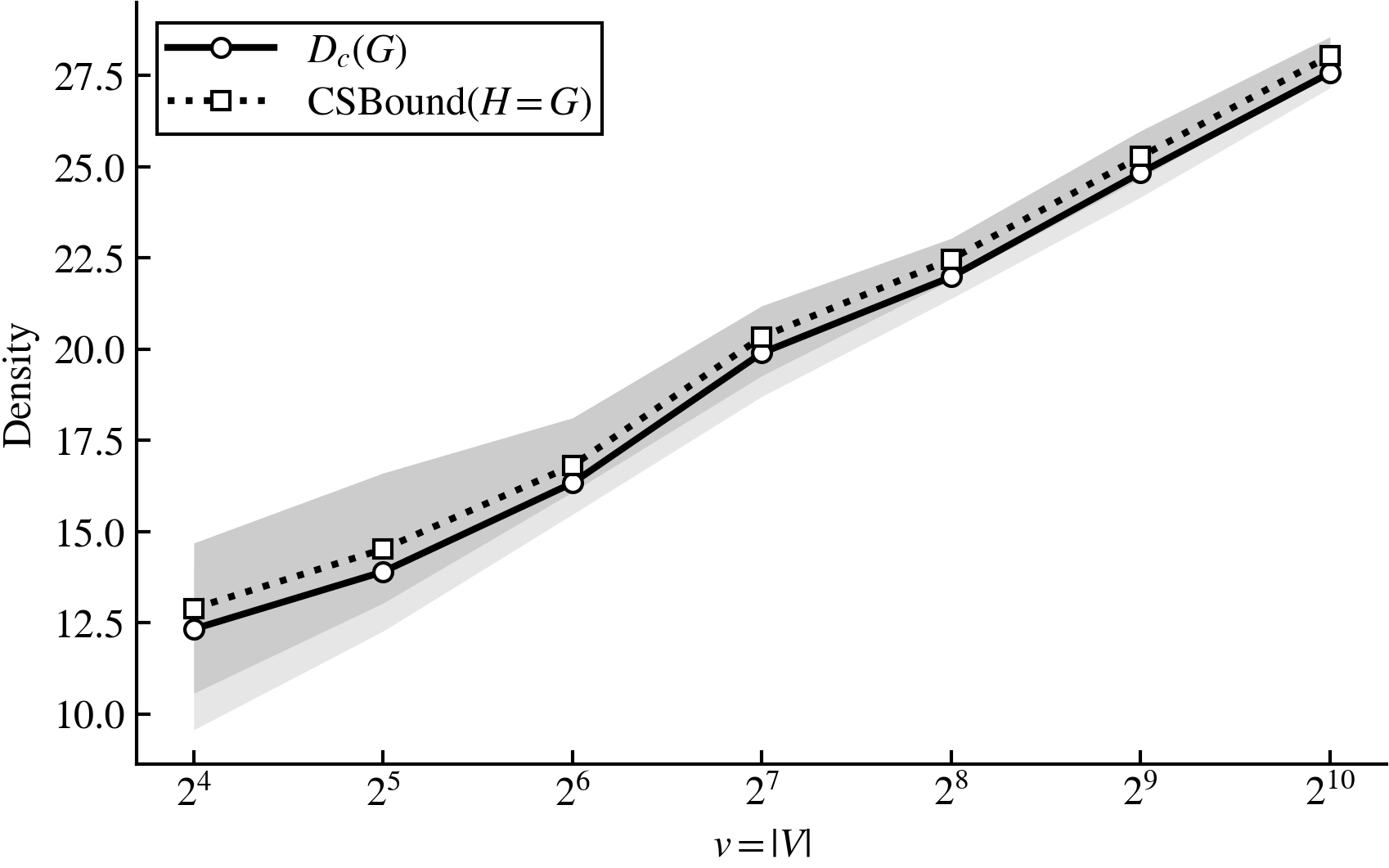}
    \caption{Densities with heterogeneous $c$ and $\lambda\in [1, 10]$}
    \label{fig:placeholder}
\end{figure}

\begin{figure}[H]
    \centering
    \includegraphics[width=0.85\linewidth]{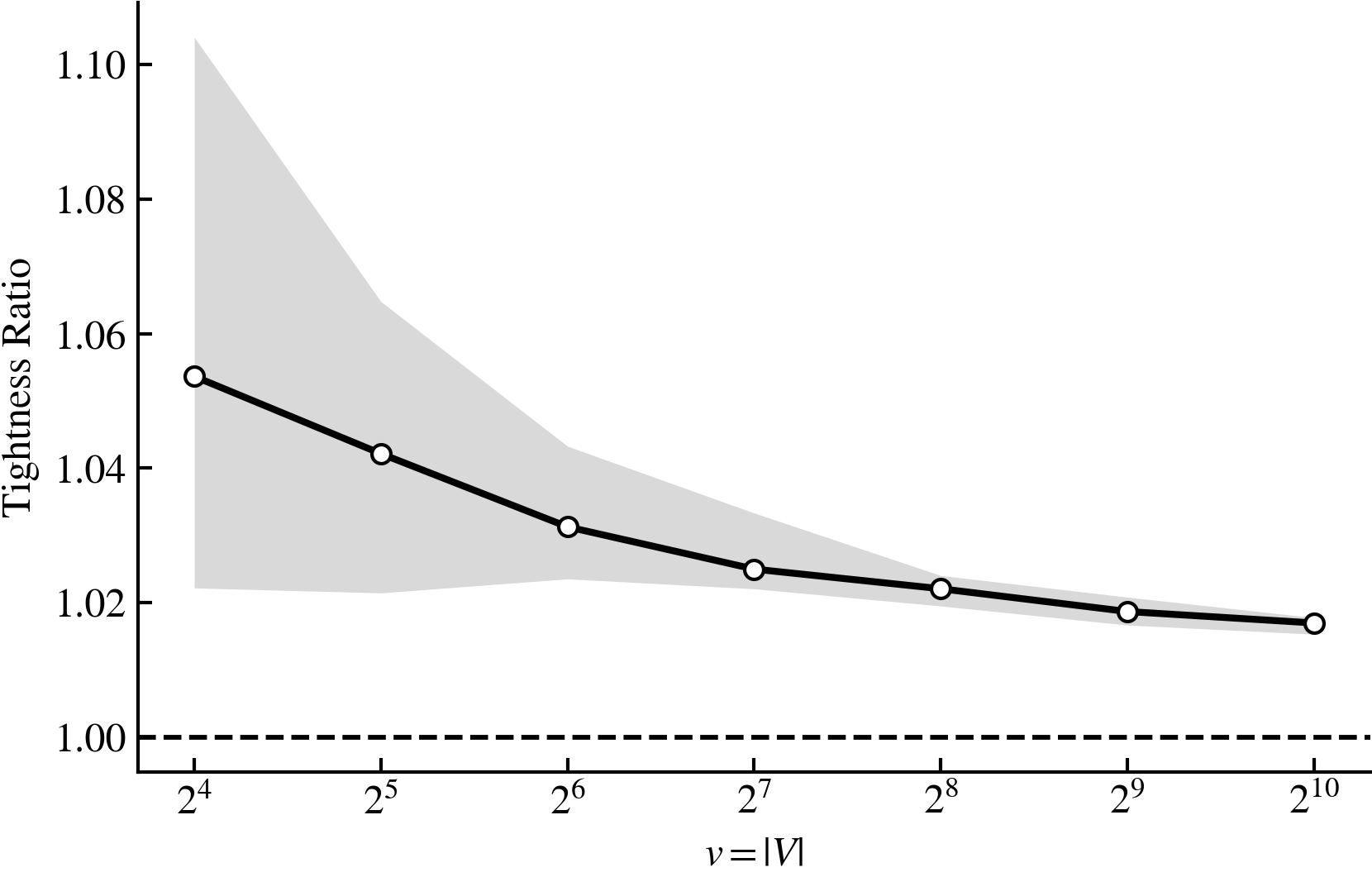}
    \caption{Tightness with heterogeneous $c$ and $\lambda\in [1, 10]$}
    \label{fig:placeholder}
\end{figure}

\begin{figure}[H]
    \centering
    \includegraphics[width=0.85\linewidth]{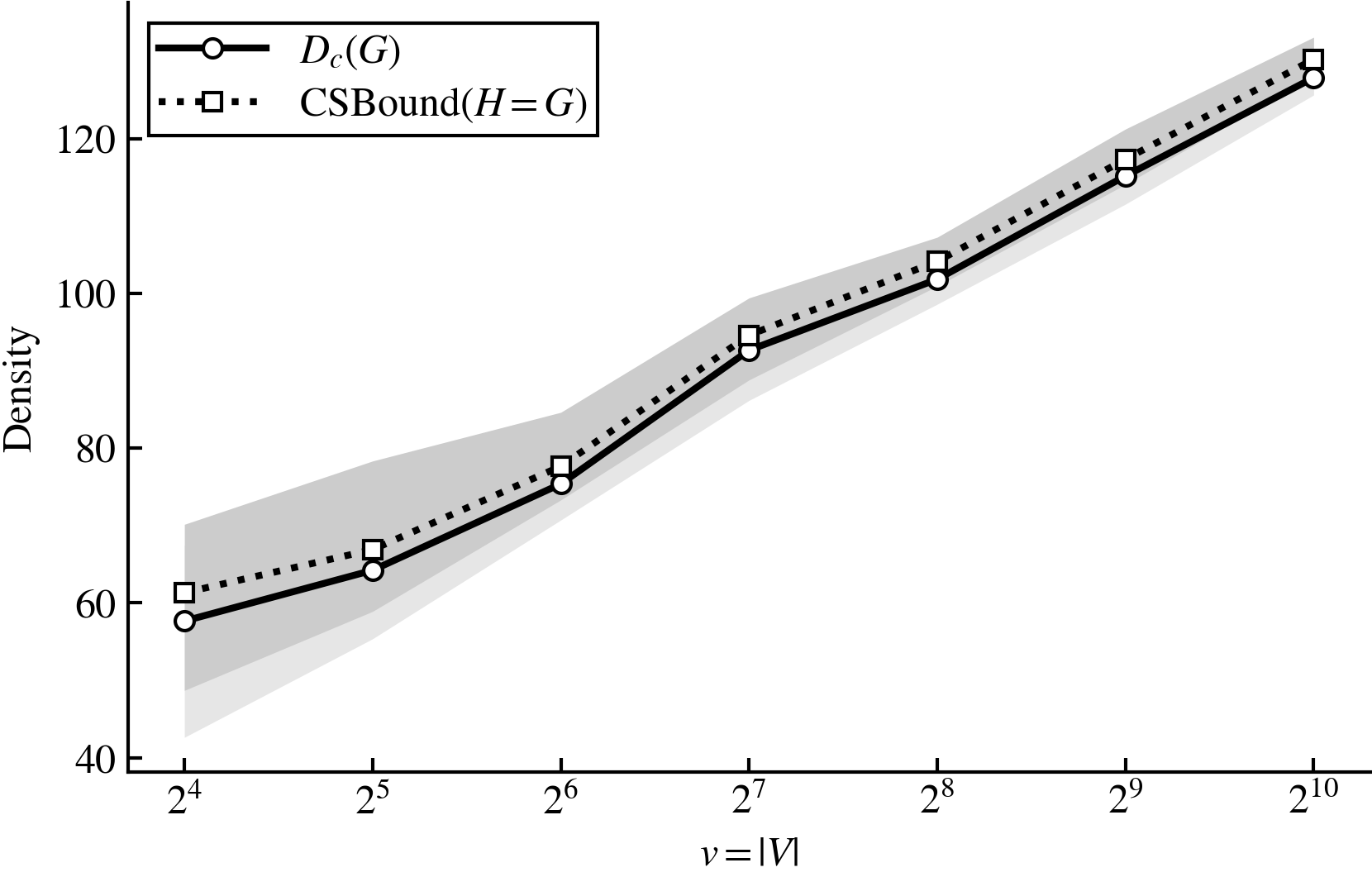}
    \caption{Densities with heterogeneous $c$ and $\lambda\in [1, 50]$}
    \label{fig:compute1}
\end{figure}

\begin{figure}[H]
    \centering
    \includegraphics[width=0.85\linewidth]{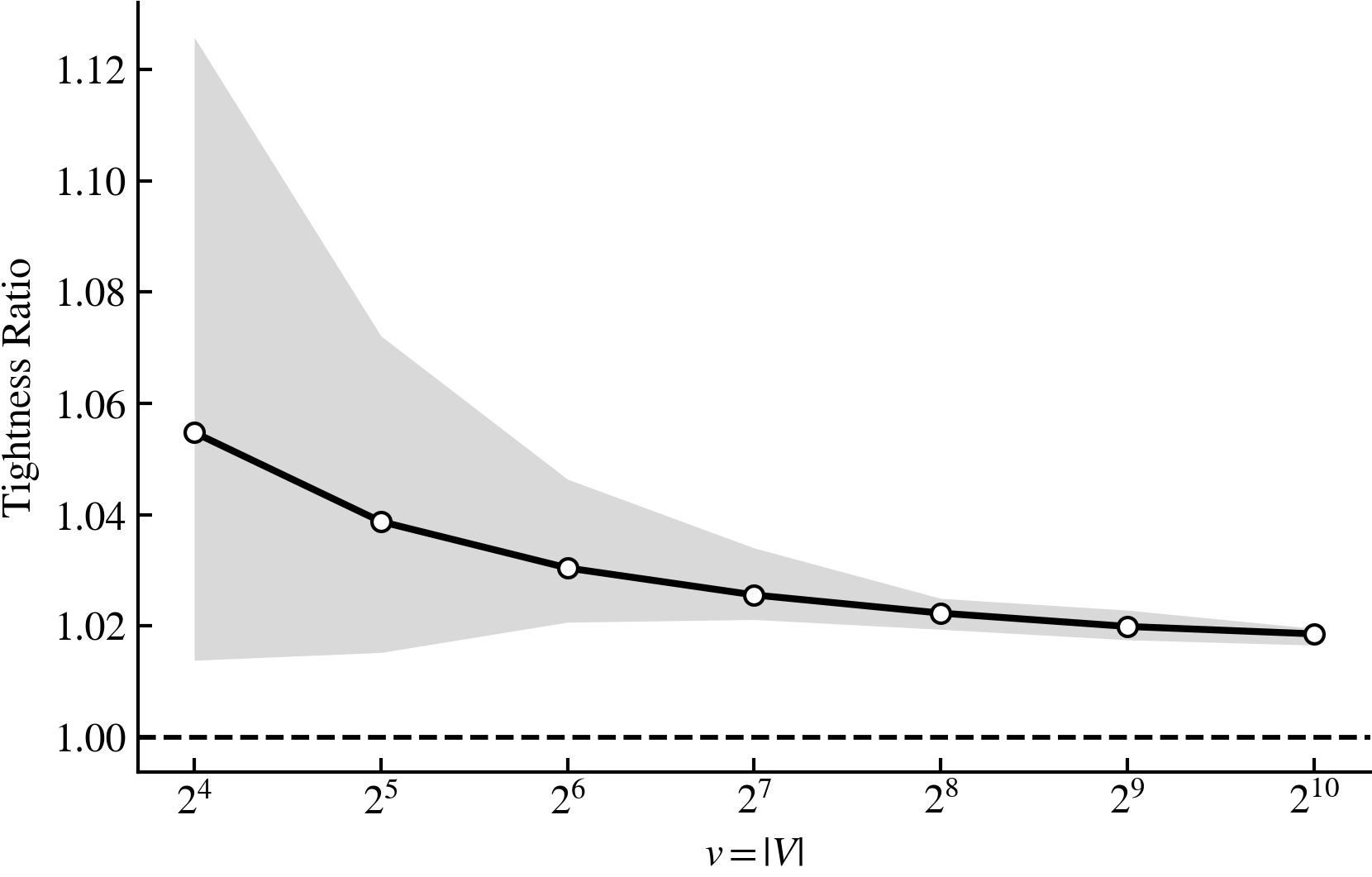}
    \caption{Tightness with heterogeneous $c$ and $\lambda\in [1, 50]$}
    \label{fig:compute1}
\end{figure}
\begin{remark}
    Indeed, the bound is tight at asymptotically moderately large $v$; further research is needed to fully explore the computational complexity of different graph families. 
\end{remark}
In short, the hypercube family provides a clean stress-test where symmetry allows us to reduce the evaluation of the certificate to a single grounded Laplacian solve (up to relabeling), while still allowing for heterogeneous conductance assignments through random sampling. The overlay figures indicate that $\mathrm{CSBound}(H{=}Q_d)$ tracks the density functional $D_c(Q_d)$ closely, and the tightness ratios remain close to $1$ even as $|V(Q_d)|=2^d$ grows (in particular, the tightness ratio never dips below 1 (Proposition~\ref{prop:cs-upper})). These experiments suggest that the certificate is not just an artifact of unit weights, but also behaves stably under significant weight heterogeneity, motivating the broader family-wise computations and complexity discussion in future projects.

\section{Conclusion}\label{sec:conclusion}
This paper introduced and studied a conductance--weighted analogue of arboricity built from the density functional
\[
D_c(H)=\frac{\sum_{e\in E(H)} c(e)}{|V(H)|-1},
\]
together with the associated extremal quantity obtained by maximizing over connected subgraphs.
The guiding theme is that embedding conductances into the unweighted fractional arboricity framework produces an invariant that is
both structurally similar to fractional arboricity and also naturally tethered to effective resistance in electrical network theory.
On the theoretical side, we developed a usable upper bound for the extremal density in terms of effective resistances in the ambient graph, which is explicit once effective resistances are available.
The resulting inequality behaves well under scaling, operates cleanly with classical families, and is computable in practice via standard linear algebra. A brief computational exhibit complements the theory. Using Python, we generated overlay and tightness-ratio plots comparing the baseline density
to the explicit certificate across graph sizes. We emphasized the hypercube family \(Q_d\), where edge--transitivity implies that all edges share the same effective resistance,
so \(\mathrm{CSBound}(H{=}G)\) can be evaluated from a single grounded--Laplacian solve per dimension. Both in the all-ones base case and under random two-point conductances, the certificate is consistently similar in value to the densities, and the tightness ratio trends toward \(1\) as \(d\) grows, supporting the view that the bound is sharp on highly symmetric families.

There are two natural directions for future work. First, we pursue an economics application in which conductances encode intensity or reliability of interactions in a network (such as global trade or social contagion), so that the weighted arboricity functional measures a density that
incorporates both combinatorial connectivity and resistance geometry. Second, on the purely mathematical side, the objects introduced here invite a more systematic analysis and algebraic formalization: finer structural properties of the extremal functional or limiting behavior under graph operations. We expect these two strains, applied network modeling and deeper structural theory, to reinforce each other with computation serving as a useful tool in that pursuit.


\end{document}